\documentclass{amsart}[12pt]
\usepackage{amsmath, amsthm, amscd, amsfonts,verbatim}
\usepackage{graphicx,float}

\usepackage{amssymb}

\newcommand{\forces}{\mathrel\Vdash}

\DeclareMathOperator{\dom}{dom}

\DeclareMathOperator{\ran}{ran}
\DeclareMathOperator{\diam}{diam}

\def\MPB{{\mathbb{P}}}

\def\MRB{{\mathbb{R}}}
\def\MCB{{\mathbb{C}}}

\def\k{\kappa}

\def\lan{\langle}
\def\ran{\rangle}
\def\a{\alpha}
\def\om{\omega}

\def\ov{\overline}

\newcommand{\norm}[1]{\lVert#1\rVert}

\newcommand{\explicitSet}[1]{\left\lbrace #1 \right\rbrace}
\newcommand{\Set}[2]{\explicitSet{#1 \colon #2}}
\newcommand{\rest}{\!\restriction\!}
\newcommand{\tr}[1]{[\![#1]\!]}

\newtheorem{theorem}{Theorem}[section]
\newtheorem{lemma}[theorem]{Lemma}

\newtheorem{corollary}[theorem]{Corollary}
\newtheorem{remark}[theorem]{Remark}

\newtheorem{question}[theorem]{Question}
\numberwithin{equation}{section}

\usepackage{pb-diagram}

\def\lan{\langle}
\def\ran{\rangle}
\def\ov{\bar}
\def\rmark{\mbox{$\rm\bf\rule{0.06em}{1.45ex}\kern-0.05em R$}}
\def\pmark{\mbox{$\rm\bf\rule{0.06em}{1.45ex}\kern-0.05em P$}}
\def\nmark{\mbox{$\rm\bf\rule{0.06em}{1.45ex}\kern-0.05em N$}}
\def\vdash{\mbox{$\rm\| \kern-0.13em -$}}

\def\lan{\langle}
\def\ran{\rangle}
\def\ov{\bar}
\def\rmark{\mbox{$\rm\bf\rule{0.06em}{1.45ex}\kern-0.05em R$}}
\def\pmark{\mbox{$\rm\bf\rule{0.06em}{1.45ex}\kern-0.05em P$}}
\def\nmark{\mbox{$\rm\bf\rule{0.06em}{1.45ex}\kern-0.05em N$}}
\def\vdash{\mbox{$\rm\| \kern-0.13em -$}}
\newcommand{\lusim}[1]{\smash{\underset{\raisebox{1.2pt}[0cm][0cm]{$\sim$}}
{{#1}}}}

\begin{document}

\title[Combinatorial and number-theoretic properties of generic reals]{Combinatorial and number-theoretic properties of generic reals}

\author[W. Brian and M. Golshani ]{Will Brian and Mohammad
  Golshani }


\thanks{The second author's research has been supported by a grant from IPM (No. 99030417).}

\thanks{The authors thank Alex Kruckman for bringing the Theorem of Friedman-Miller \ref{friedman-miller theorem}(a) to their attention.} \maketitle
%{ \\ Department of Mathematics\\  Shahid Bahonar University of Kerman, Kerman, Iran}

%\subjclass[2010]{03E35}

%\keywords{ Cohen and random reals, partition regular families, Mahler's classification of reals, Hausdorff dimension, o-minimality}
\begin{abstract}
We discuss  some properties of Cohen and random reals. We show that they belong to any definable partition regular family, and hence they satisfy most ``largeness'' properties studied in Ramsey theory. We determine their position in the Mahler's classification of the reals and using it, we get some information about Liouville numbers.  We also show that they are wild in the sense of o-minimality, i.e., they define the set of integers.
\end{abstract}

\section{introduction}
In this paper we study Cohen and random reals, and ask questions about what sort of reals they are. Do they satisfy the conclusion of Hindman's theorem? How badly approximable are they by algebraic numbers?
%In $V[r]$ (where $r$ is random), what is the Hausdorff dimension of the set of all random-over-$V$ reals, or of $\mathbb R^V$?
What of the same questions, but with Cohen reals?
Are Cohen or random reals wild in the sense of o-minimality?

Let $\mathbb N$ denote the natural numbers (without $0$). We say that $\mathcal F$ is a \emph{semifilter} of infinite subsets of $\mathbb N$ if it is closed under taking supersets. Such collections are also sometimes called \emph{Furstenberg families}, or simply \emph{families}. A semifilter $\mathcal F$ is \emph{partition regular} if for every $A \in \mathcal F$, and every partition of $A$ into finite sets $A_1,A_2,\dots,A_n$, there is some $i \leq n$ such that $A_i \in \mathcal F$. Roughly, a partition regular semifilter is a collection of sets that can be considered not small. Such collections are also sometimes called \emph{co-ideals}.

For example, call $r \subseteq \mathbb N$ \emph{large} if $\Sigma_{n\in r}$$ 1\over n$$=\infty.$ A famous conjecture
of Erd\"{o}s--Tur\'{a}n says that if $r$ is large, then it contains arbitrary long arithmetic progressions. Both the semifilter of large sets and the semifilter of sets containing arbitrarily long arithmetic progressions are partition regular. (The second assertion is the famous theorem of van der Waerden). %The second author's musings on this conjecture, and how it might relate to forcing (see, e.g., \cite{}), were the starting point for the investigations leading to this paper.

As it turns out, Cohen and random reals belong to essentially every partition regular family considered in classical Ramsey theory. Specifically, we prove the following:

\begin{theorem}\label{largeness}
Let $V \subseteq W$ be models of set theory, and let $r \in W$ be a Cohen (or random) real over $V$.
\begin{enumerate}
\item Suppose $W=V[r]$. In $W$, $r$ belongs to every partition regular family that is definable over $V$.
\item In general, $r$ belongs to every partition regular family that is lightface $\Sigma^1_2(a)$ or $\Pi^1_2(a)$-definable, where $a$ is a real number in $V$.
\end{enumerate}
\end{theorem}

We will provide an example showing that the conclusion of $(1)$ does not hold under the hypotheses of $(2)$; in other words, some further restriction on the niceness of the partition regular families is really needed for $(2)$.

%\begin{corollary}
%If $r$ is a Cohen or random real, then
%\begin{enumerate}
%\item $r$ is large.
%\item $r$ contains arbitrary long arithmetic progressions.
%\item $r$ contains arbitrary long geometric progressions.
%\item $r$ contains infinitely many solutions to every partition regular Diophantine equation.
%\item $r$ has positive upper density.
%\item $r$ is piecewise syndetic, i.e.,  there are arbitrarily long intervals of $\mathbb {N}$  where the gaps in $r$ are bounded by some constant $b$.
%\item $r$ satisfies the conclusion of Hindman's Theorem: i.e., there is some infinite $A \subseteq r$ such that, for every finite $F \subseteq A$, $\sum F \in r$.
%\item $r$ is a $\Delta$-set: i.e., $r$ contains $\{s_j-s_i : i,j \in \mathbb N, i < j\}$ for some infinite sequence $\langle s_i : i \in \mathbb N \rangle$ of natural numbers.
%\item $r$ is central: i.e., there is a minimal idempotent ultrafilter containing $r$.
%\end{enumerate}
%\end{corollary}

Then we consider Mahler's classification of the reals.
The set of real numbers splits into algebraic and transcendental numbers. Kurt Mahler in 1932 partitioned the real numbers into four classes, called $A, S, T,$ and $U$. Class $A$ consists of algebraic numbers, while classes $S, T$
and $U$ contain transcendental numbers. We determine the situation of Cohen and random reals by proving the following:
\begin{theorem} (\cite{dean})
\label{mahler}$\ $

$(a)$
Assume $r$ is a Cohen real over some (possibly countable) transitive model. Then $r$ is in the class $U$ (and indeed it is a Liouville number).

$(b)$ Assume $r$ is a random real over some (possibly countable) transitive model. Then $r$ is in the class $S$.
\end{theorem}

%We also consider the Hausdorff dimension of the set of Cohen or random reals and prove the following:
%\begin{theorem}
%\label{dimension}
%$(a)$ Assume $V[G]$ is a generic extension of $V$ by Cohen forcing and (in $V[G]$) set
%\begin{center}
%$S= \{ r \in [0,1]: r$ is Cohen generic over $V              \}$.
%\end{center}
%Then  the Hausdorff dimension of $S$ is equal to
%$0$. The same result holds for the set of old reals in the interval $[0, 1].$

%$(b)$ Assume $V[G]$ is a generic extension of $V$ by random forcing and (in $V[G]$) set
%\begin{center}
%$S= \{ r \in [0,1]: r$ is random generic over $V              \}$.
%\end{center}
%Then  the Hausdorff dimension of $S$ is equal to
%$1.$ The same result holds for the set of old reals in the interval $[0, 1].$
%\end{theorem}

We show also that Cohen or random reals  are wild in the sense of o-minimality,
by proving the following.
\begin{theorem}
\label{o-minimality}
 Assume $r$ is a Cohen (a random) real over $V$. Then $r$ defines $\mathbb{Z}$ in the sense that the set of integers $\mathbb{Z}$ is definable in the structure
 $(\mathbb{R}^{V[r]}, +, \cdot, <, 0, 1, r)$, where $r$ is considered as a unary predicate.
\end{theorem}
We also show that the above phenomenon is not true for all generic reals by providing an example of a generic real $r$ which does not define $\mathbb{Z}.$

Finally we consider the situation of adding many Cohen or random reals and extend a classical result by proving the following:
\begin{theorem}
	\label{adding many}
	Let $\kappa$ be an infinite cardinal.  Then, forcing with $\MRB(\kappa) \times \MRB(\kappa)$ adds a generic filter for $\MCB(\kappa),$
	where  $\MRB(\kappa)$ and $\MCB(\kappa)$ are the forcing notions for adding $\kappa$-many random reals and  adding $\kappa$-many Cohen reals respectively.
\end{theorem}
\section{Cohen and random forcing}

In this section, we briefly review the definition and basic properties of the Cohen and random forcing notions.

%Since Borel sets are constructed in countably many stages, beginning with countably many rational intervals, and at every stage taking complements or countable unions, every projective set can be ``coded'' by a function $\omega \rightarrow \{0,1\}$. In fact, the same is true of projective sets: the only difference is that some of the stages of construction for a projective set might involve taking continuous images, but of course a continuous function is determined by its action on $\mathbb Q$, and so it too carries only countably many bits of information.

%This coding allows us to take descriptions of projective sets in the ground model and reinterpret them in a generic extension. For sets appearing early enough in the projective hierarchy (specifically, $\Sigma^1_2 \cup \Pi^1_2$), including all Borel sets, many properties of these reinterpretations are absolute (properties like $A \subseteq B$, $A \neq \emptyset$, or $\mu(A) = c$).

\subsection{Cohen forcing} The Cohen forcing, denoted  $\MCB,$ is defined as
\[
\MCB=\{p: p\text{~is a function~}n \to 2 \text{ for some } n \in \omega        \},
\]
ordered by reverse inclusion. Suppose $G$ is $\MCB$-generic over $V$ and let $F=\bigcup_{p \in G}p.$ Then $F$ is a function from $\omega$
into $2$ and if we set $r=\{n \in \omega: F(n)=1       \}$, then $V[G]=V[r].$ The real $r$ is called the Cohen real. This real has a canonical $\mathbb C$-name $\dot r$, determined by
$$\norm{\dot r(n) = 1} = \sum \Set{p \in \mathbb C}{p(n)=1}.$$
The next lemma is a well-known characterization of Cohen reals which we apply repeatedly.
\begin{lemma}
\label{cohen characterization}
A real $r$ is Cohen over $V$ if and only if it does not belong to any meager Borel set coded in $V$.
\end{lemma}

 In general, given a non-empty set $I$, the forcing notion $\MCB(I)$, the Cohen forcing for adding $|I|$-many Cohen
 reals is defined by
 \begin{center}
 	$\MCB(I)=\{p: I \times \omega \to 2: |p|< \aleph_0                      \}$,
 \end{center}
 which is ordered by reverse inclusion.
 \begin{lemma}
 	\label{chain condition lemma for cohen}
 	$\MCB(I)$ is c.c.c.
 \end{lemma}
 Assume $G$ is $\MCB(I)$-generic over $V$, and set $F=\bigcup G: I \times \omega \to 2.$ For each $i \in I$ define $c_i: \omega \to 2$
 by $c_i(n)=F(i, n).$ Then
 	for  each $i \in I, c_i \in 2^\omega$
 	is a new real and for $i \neq j$ in $I, c_i \neq c_j$. Furthermore, $V[G]=V[\langle  c_i: i \in I    \rangle].$
 By $\kappa$-Cohen reals over $V$, we mean a sequence  $\langle c_i: i< \kappa  \rangle $
 which is $\MCB(\kappa)$-generic over $V$.

\subsection{Random forcing} There are several equivalent ways to present the random forcing. Here, we consider the following presentation.
Consider the product measure space $2^{\omega}$ with the standard product measure $\mu$ on it.  Let $\mathbb{B}$ denote the class of Borel subsets of $2^{\omega}$.

Note that sets of the form
\[
\tr{s}=\{x \in 2^{ \om}: x \upharpoonright n=s              \},
\]
where  $s: n \to 2 $ is a finite partial function
form a basis of open sets of $2^{ \om}.$

For Borel sets $S, T \in \mathbb{B}$ set
\[
S \sim T \iff S \bigtriangleup T \text{~is null,~}
\]
where $S \bigtriangleup T$ denotes the symmetric difference of $S$ and $T$. The relation $\sim$ is easily seen to be an equivalence relation
on $\mathbb{B}.$
Then $\MRB$, the random forcing, is defined as
\[
\MRB= \mathbb{B}/ \sim.
\]
Thus elements of $\MRB$ are equivalence classes $[B]$ of Borel sets modulo null sets. The order relation is defined by
\[
[S] \leq [T] \iff \mu(S \setminus T) =0.
\]
Let $\dot r$ be an $\MRB$-name for a real such that
$$\norm{\dot r(n) = 1} = \bigcup \Set{\tr{s}}{s(n) = 1}.$$
This real is called the random real, and this name is its canonical name.

The next lemma gives a characterization of random reals analogous to Lemma~\ref{cohen characterization}.
\begin{lemma}
\label{random characterization}
A real $r$ is random over $V$ if and only if it does not
belong to any null Borel set coded in $V$.
\end{lemma}

We think of $2^\omega$ as the end space of the tree $2^{<\omega}$. If $T$ is any subtree of $2^{<\omega}$, then $T$ defines a closed subset of $2^\omega$, namely
$$\tr{T} = \bigcap_{n \in \omega}\bigcup \Set{\tr{s}}{s \in T \text{ and } \dom(s) = n}.$$
Conversely, if $C$ is a closed subset of $2^\omega$, then there is a subtree $T$ of $2^{<\omega}$ with $\tr{T} = C$, namely
$$T = \Set{s \in 2^{<\omega}}{\tr{s} \cap C \neq \emptyset}.$$
%If $T$ is a subtree of $2^{<\omega}$ and $t \in T$, we will write
%$$T \rest t = \Set{s \in T}{\text{either } s \subseteq t \text{ or } t \subseteq s}.$$
%Thus, for example, $\tr{T \rest t} = \tr{T} \cap \tr{t}$. See the second chapter of \cite{Kechris} for more details.
\begin{lemma}
\label{random characterization 2}
Let $C \subseteq 2^\omega$ be closed and non-null, and let $T$ be the subtree of $2^{<\omega}$ with $C = \tr{T}$. Then
$[C] \forces \dot r \rest n \in T$
for every $n \in \mathbb N$.
\end{lemma}
As in the case of Cohen forcing, we can generalize the above construction to add many random reals. Suppose $I$ is a non-empty set and consider the product measure space $2^{I \times \omega}$ with the standard product measure $\mu_I$ on it.  Let $\mathbb{B}(I)$ denote the class of Borel subsets of $2^{I \times \omega}$.
Recall that $\mathbb{B}(I)$ is the $\sigma$-algebra generated by the basic open sets
\[
\tr{s}=\{ x \in  2^{I \times \omega}: x \supseteq s             \},
\]
where $s \in \MCB(I)$. Also $\mu_I(\tr{s})= 2^{-|s|}$.

For Borel sets $S, T \in \mathbb{B}(I)$ define $S \sim T $ as above.
Then $\MRB(I)$, the forcing for adding $|I|$-many random reals, is defined as

\[
\MRB(I)= \mathbb{B}(I) / \sim.
\]
The following fact is standard.
\begin{lemma}
	\label{chain condition lemma for random}
	$\MRB(I)$ is c.c.c. for any $I$.
\end{lemma}
Using the above lemma, we can easily show that $\MRB(I)$ is in fact a complete Boolean algebra.
Given any $i \in I$, let $\dot{r}_i$ be an $\MRB(I)$-name for a real such that
$$\norm{\dot{r}_i(n) = 1} = \bigcup \Set{\tr{s}}{s(i, n) = 1}.$$
Given an $\MRB(I)$-generic filter $G$ over $V$, and $i \in I$, set $r_i = \lusim{r}_i[G].$ Then each $r_i \in 2^\omega$
	is a new real and for $i \neq j$ in $I, r_i \neq r_j$. Furthermore, $V[G]=V[\langle  r_i: i \in I    \rangle].$
By $\kappa$-random reals over $V$, we mean a sequence  $\langle r_i: i< \kappa  \rangle $
which is $\MRB(\kappa)$-generic over $V$.

\section{Largeness properties of Cohen and random reals}

In this section we prove Theorem \ref{largeness}. The proof of part $(1)$ capitalizes on the existence of certain automorphisms of the Cohen and random forcing notions (and then $(2)$ follows from $(1)$ by absoluteness). We begin by establishing the existence of the relevant automorphisms.

\begin{lemma}
Let $\mathbb P$ denote either the forcing for adding a Cohen real or a random real, and let $\dot r$ denote the canonical name for the generic real. Then, for any $p \in \mathbb P$, there is an automorphism $\pi: \mathbb P \rightarrow \mathbb P$ and some $q \leq p$ such that
\begin{itemize}
\item $\pi(q) = q$.
\item $q \forces \dot r \cap \pi(\dot r) \text{ is finite}$.
\item For some finite set $F$, $q \forces \dot r \cup \pi(\dot r) \cup F=\omega$.
\end{itemize}
\end{lemma}
\begin{proof}
For the Cohen forcing the proof is straightforward. Let $p \in \mathbb C$ and let $\pi: \mathbb C \rightarrow \mathbb C$ be the map defined by
\begin{itemize}
\item $\pi(s)(n) = s(n)$ for all $n \in \dom(p)$
\item $\pi(s)(n) = 1-s(n)$ for all $n \notin \dom(p)$.
\end{itemize}
Let $q=p$. It is clear that $\pi(q) = q$, and if $\dot r$ is the canonical name for the generic real, then for all $n \notin \dom(q)$,
$$\forces_{\mathbb P} \ \dot r(n) = 1 \ \Leftrightarrow \ \pi(\dot r)(n) = 0$$
which in particular implies that $q \forces \dot r \cap \pi(\dot r) \subseteq \dom(q)$ and $r \cup \pi(\dot r) \cup \dom(q) = \omega$.

For the random forcing, let $p = [B] \in \mathbb R$. By the Lebesgue Density Theorem (see Exercise 17.9 in \cite{Kechris}), there is some $s \in 2^{<\omega}$ such that
$$\mu(B \cap \tr{s}) > \frac{1}{2}\mu(\tr{s}).$$

Define a map $\gamma: 2^{\omega} \rightarrow 2^{\omega}$ as follows:
\[
\gamma(x)(n) =
\begin{cases}
x(n) & \text{if } n \in \dom(s) \\
1 - x(n) & \text{otherwise.}
\end{cases}
\]
Observe that $\gamma$ is a homeomorphism from $2^\omega$ to itself; in particular, it maps Borel sets to Borel sets. Also observe that $\mu(\gamma(\tr{t})) = \mu(\tr{t})$ for every basic open set $\tr{t}$. It follows that $\gamma$ is measure-preserving: i.e., $\mu(A) = \mu(\gamma(A))$ for every measurable set $A$. Also notice that $\gamma(\tr{s}) = \tr{s}$.

Let $C = (\tr{s} \cap B) \cap \gamma(\tr{s} \cap B)$. By the observations at the end of the last paragraph, $C$ is Borel and
$$\mu(C) \geq \mu(\tr{s} \cap B) - \mu(\tr{s} - B),$$
and this is positive by our choice of $s$. Thus $[C] \in \mathbb R$, and clearly $[C] \leq [B]$. Let $q = [C]$.

Define $\pi_0: 2^\omega \rightarrow 2^\omega$ by
\[
\pi_0(x) =
\begin{cases}
\gamma(x) & \text{if } x \in C \\
x & \text{if } x \notin C.
\end{cases}
\]

First, note that $\pi_0$ maps Borel sets to Borel sets. Second, $\pi$ is measure-preserving:
$$\mu(\pi_0(A)) = \mu(\pi_0(A \cap C) \cup (A-C)) =  \mu(\gamma(A \cap C))+\mu(A-C) = \mu(A \cap C) +\mu(A-C)=\mu(A).$$
Third, $\pi_0$ is a bijection.
Fourth, $\pi_0$ is order-preserving on $\mathbb R$.

Therefore $\pi_0$ induces an automorphism $\pi$ of $\mathbb R$, namely $\pi([A]) = [\pi_0(A)]$. Let us check that it has the required properties. Because $\gamma \circ \gamma$ is the identity map, it follows from our definition of $C$ that $\gamma(C) = C$. Therefore $\pi_0(C) = C$ and $\pi(q) = q$. Furthermore, if $T$ is the subtree of $2^{<\omega}$ with $\tr{T} = C$, we have
$$q \forces \dot r \rest n \in T.$$
If $t \in T$ and $s \subseteq t$, then $\pi_0(\tr{t}) = \gamma(\tr{t}) = \tr{u}$, where
\[
u(n) =
\begin{cases}
t(n) & \text{if } n \in \dom(\sigma) \\
1 - t(n) & \text{otherwise.}
\end{cases}
\]
Thus, for all $n \notin \dom(s)$, we have
$$q \forces \dot r(n) \neq \pi(\dot r)(n).$$
It follows that
$$q \forces \ \dot r \cap \pi(\dot r) \subseteq \dom(s) \ \ \text{ and } \ \ \dot r \cup \pi(\dot r) \cup \dom(s) = \omega,$$
which completes the proof.
\end{proof}

\begin{proof}[Proof of Theorem~\ref{largeness}]

Let $\mathcal F$ be a partition regular semifilter in $W$ that is definable over $V$. In other words, there is a formula $\varphi(X,a)$, with $a \in V$, such that (in $W$) if $X \subseteq \mathbb N$ then
$$X \in \mathcal F \ \Leftrightarrow \ \varphi(X,a)$$

Assume $W = V[r]$ for some Cohen (random) real $r$. Let $\mathbb P$ denote the Cohen (random) forcing, so that $W$ is a $\mathbb P$-generic extension of $V$, and let $\dot r$ denote the canonical name for the generic real. Let $G$ be the generic filter on $\mathbb P$ with $\dot r^G = r$.

Let $\dot{\mathcal F}$ be the canonical $\mathbb P$-name for the semifilter $\mathcal F$:
$$\norm{\dot x \in \dot{\mathcal F}} = \norm{ \dot x \subseteq \mathbb N} \wedge \norm{\varphi(\dot x, a)}.$$
There is some $p \in G$ such that
$$p \forces \dot{\mathcal F} \text{ is a partition regular family}$$
(in fact, one may argue using automorphisms of $\mathbb P$ that if $p \forces (\dot{\mathcal F} \text{ is a partition regular family})$, then $\forces_{\mathbb P} (\dot{\mathcal F} \text{ is a partition regular semifilter})$. We do not need this, though, and content ourselves to work below $p$.)

%There is an automorphism $\pi: \mathbb P \rightarrow \mathbb P$ in $V$ mapping $p$ to $1-p$. This automorphism of $\mathbb P$ extends to an automorphism of $\mathbb P$-names, and, because $\pi(\check a) = \check a$, we have
%$$p \forces \dot x \in \dot{\mathcal F} \ \Leftrightarrow \ \pi(p) \forces \pi(\dot x) \in \dot{\mathcal F}.$$
%From the definition of a partition regular family, one may conclude that
%$$\pi(p) = 1-p \forces \dot{\mathcal F} \text{ is a partition regular family}.$$
%It follows that $\forces_{\mathbb P} (\dot{\mathcal F} \text{ is partition regular})$.

Let us suppose $r \notin \mathcal F$ and aim for a contradiction. If $r \notin \mathcal F$, there is some $p' \leq p$, with $p' \in G$, such that $p' \forces \dot r \notin \dot{\mathcal F}$ or, equivalently,
$$p' \forces \neg \varphi(\dot r,a).$$

Applying our lemma, let $\pi$ be an automorphism of $\mathbb P$ and $q \leq p'$ such that
\begin{itemize}
\item $\pi(q) = q$.
\item $q \forces \dot r \cap \pi(\dot r) \text{ is finite}$.
\item $q \forces \dot r \cup \pi(\dot r) \cup F = \omega$ for some finite set $F$.
\end{itemize}

Because $q \leq p'$,
$$q \forces \dot{\mathcal F} \text{ is a partition regular family and } \neg \varphi(\dot r, a).$$
By well-known properties of automorphisms, we have
$$\pi(q) = q \forces \dot{\mathcal F} \text{ is a partition regular family and } \neg \varphi(\pi(\dot r), a).$$

Let $G'$ be a $\mathbb P$-generic filter containing $q$ (note that $q$ is not necessarily in $G$). Working in $W' = V[G']$, we will obtain a contradiction. By our choice of $\pi$ and the fact that $q \in G'$, we have $\omega  = (\dot r)^{G'} \cup (\pi(\dot r))^{G'} \cup F$ for some finite set $F$. Recalling our definition of partition regularity, one of $(\dot r)^{G'}$, $(\pi(\dot r))^{G'}$, and $F$ should be in $\mathcal F' = (\dot{\mathcal F})^{G'}$. Our definition of ``semifilter" requires that $\mathcal F'$ contain only infinite sets, so we must have either $(\dot r)^{G'}$ or $(\pi(\dot r))^{G'}$ in $\mathcal F'$. Because $q \in G'$ and $q \forces \neg \varphi(\dot r,a)$, we have $(\dot r)^{G'} \notin \mathcal F'$. Because $q \in G'$ and $q \forces \neg \varphi(\pi(\dot r),a)$, we have $(\pi(\dot r))^{G'} \notin \mathcal F'$. This contradiction establishes that $r \in \mathcal F$ and finishes the proof of $(1)$.

For $(2)$, let $r$ be Cohen (random) over $V$, and let $\varphi(x,a)$ be a $\Sigma^1_2(a)$ or $\Pi^1_2(a)$ formula defining $\mathcal F$ in $W$ (where $a$ is a real number in $V$). By Shoenfield's Absoluteness Theorem (which implies that such formulas are absolute), if $X \in V[r]$ then $W \models \varphi(X,a)$ if and only if $V[r] \models \varphi(X,a)$. Therefore $\varphi$ still defines a partition regular family in $V[r]$. By $(1)$, $V[r] \models \varphi(r,a)$. By Shoenfield's Absoluteness Theorem again, $W \models \varphi(r,a)$, and it follows that $r \in \mathcal F$.
\end{proof}

\begin{corollary}
Let $M \subseteq V$ be any transitive model of (enough of) ZFC. If $r$ is a Cohen or random real over $M$, then
\begin{enumerate}
\item $r$ is large.
\item $r$ contains arbitrary long arithmetic progressions.
\item $r$ contains arbitrary long geometric progressions.
\item $r$ contains infinitely many solutions to every partition regular Diophantine equation.
\item $r$ has positive upper density.
\item $r$ is piecewise syndetic, i.e.,  there are arbitrarily long intervals of $\mathbb {N}$  where the gaps in $r$ are bounded by some constant $b$.
\item $r$ satisfies the conclusion of Hindman's Theorem: i.e., there is some infinite $A \subseteq r$ such that, for every finite $F \subseteq A$, $\sum F \in r$.
\item $r$ is a $\Delta$-set: i.e., $r$ contains $\{s_j-s_i : i,j \in \mathbb N, i < j\}$ for some infinite sequence $\langle s_i : i \in \mathbb N \rangle$ of natural numbers.
\item $r$ is central: i.e., there is a minimal idempotent ultrafilter containing $r$.
\end{enumerate}
\end{corollary}
\begin{proof}
It is not difficult to check that the various families described in $(1)$ - $(8)$ are all analytic (they are, at worst, $G_\delta$, $G_\delta$, $G_\delta$, $F_\sigma$, $F_\sigma$, $G_{\delta \sigma}$, $\Sigma^1_1$, and $\Sigma^1_1$, respectively). For $(9)$, observe that, though the central sets are often defined as those belonging to some minimal idempotent ultrafilter, there are alternative definitions (see, e.g., Theorem 14.25 in \cite{H&S}) that make it clear that the family of central sets is at worst $\Sigma^1_1$.
\end{proof}

Next, let us consider an example showing that the conclusions of $(1)$ do not hold for all extensions $W$ of $V$; in particular, some restriction on the definition of $\mathcal F$ is required.

Given $V$, let $r$ be a Cohen real over $V$, and let $W$ be the generic extension of $V[r]$ obtained by using a (set-sized) Easton forcing to get
$$n \in r \ \ \Leftrightarrow \ \ 2^{\aleph_n} = \aleph_{n+1}.$$
In $W$, we may define a partition regular semifilter $\mathcal F$ by the parameter-free formula
$$X \in \mathcal F \ \ \Leftrightarrow \ \ X \subseteq \mathbb N \text{ and }(\forall m \in \mathbb N)(\exists n \in X) \ n > m \ \wedge \  2^{\aleph_n} \neq \aleph_{n+1}.$$
In other words, $\mathcal F$ is the family of all sets that have infinite intersection with the complement of $r$. Clearly $\mathcal F$ is partition regular, but $r \notin \mathcal F$.

In fact,  one can show that
Theorem \ref{largeness}(2) can not be improved. To see this, assume $V=L[r],$ where $r$ is a Cohen or a random real over $L$. Then there is a generic extension $W$ of $L[r]$ in which there exists a $\Delta^1_3$-real
$s$ which codes $r$ \footnote{This can be done by selectively destroying the stationarity of certain canonical stationary subsets of $\omega_1$ in $L$ and coding these stationary kills by reals in a nice way. See \cite{friedman} and \cite{holy} for similar arguments and details.}. Then the above argument shows that there exists a partition regular semifilter $\mathcal F \in W$, definable in $W$ from the parameter $s$ (hence $\Delta^1_3$),
such that $r \notin \mathcal F$.

To end this section, let us recall that there is a duality between filters and partition regular semifilters.

\begin{lemma}
If $P$ is a filter then
$$\mathcal F = \Set{X \subseteq \mathbb N}{X \cap B \neq \emptyset \text{ for all }B \in P}.$$
is a partition regular semifilter.
\end{lemma}

The family $\mathcal F$ is called the \emph{dual} of $P$ (in fact, partition regular families are sometimes referred to in the literature as \emph{filterduals}). Although we do not need to use the fact, let us point out that families can also be ``dualized'' as in the lemma above, and a family is partition regular if and only if its dual is a filter.

Theorem~\ref{largeness} admits a dual version:

\begin{theorem}\label{largeness2}
Let $V \subseteq W$ be models of set theory, and let $r$ be a Cohen (random) real over $V$.
\begin{enumerate}
\item Suppose $W=V[r]$. In $W$, $r$ belongs to no filter that is definable over $V$.
\item In general, $r$ belongs to no filter that is lightface $\Pi^1_2(a)$-definable, where $a$ is a real number in $V$.
\end{enumerate}
\end{theorem}
\begin{proof}
Let $P$ be a filter in $W$, and let $\mathcal F$ denote its dual. Clearly, $\mathcal F$ is definable over $V$ if and only if $P$ is. Also, if $P$ is lightface $\Pi^1_2(a)$-definable, where $a$ is a real number in $V$, then $\mathcal F$ is lightface $\Pi^1_2(a)$-definable as well.

By Theorem~\ref{largeness}, $\mathbb N - r$ (which is also a Cohen/random real) belongs to $\mathcal F$ under the hypotheses of either $(1)$ or $(2)$. Of course, if $\mathbb N - r$ belongs to $\mathcal F$ then $r \notin P$.
\end{proof}

Because the definition of a family from its dual requires a universal quantifier, it is possible that the dual of a $\Sigma^1_2(a)$-definable filter is neither $\Sigma^1_2(a)$ nor $\Pi^1_2(a)$. This is why we have dropped the $\Sigma^1_2(a)$ part from $(2)$ in the dual version. However, we will point out that one may go through the proof of Theorem~\ref{largeness}, dualizing everything, to obtain a stronger version of Theorem~\ref{largeness2} in which $\Sigma^1_2(a)$ filters are allowed. We leave these details as an exercise.

\section{Mahler's classification of Cohen and random reals}
In this section we consider Mahler's classification of the reals.
The definition of these classes draws and extends the idea of a Liouville number.
Below we review the definition of these classes and refer to \cite{Bugeaud} for further details. The results presented here are obtained in
\cite{dean}, where similar results are proved for some other generic reals.

For a polynomial $P(X) \in \mathbb{Z}[X]$, let $H(P)$ denote the height of $P$ and $\deg(P)$ denote the degree of $P$.
Given positive integer $n$ and real numbers $\xi$ and $H \geq 1$, define
\[
w_{n}(\xi, H)=\min\{|P(\xi)|: P(X) \in \mathbb{Z}[X], H(P) \leq H, \deg(P) \leq n, P(\xi) \neq 0                \}.
\]
Then set
\[
w_n(\xi) = \limsup_{H \to \infty} \dfrac{-\log w_n(\xi, H)}{\log H}
\]
and
\[
w(\xi) = \limsup_{n \to \infty} \dfrac{w_n(\xi)}{n}.
\]
In other words, $w_n(\xi)$ is the supremum of the real numbers $w$ for which there
exist infinitely many integer polynomials $P(X)$ of degree at most $n$ satisfying
\[
0 < |P(\xi)| < H(P)^{-w}.
\]

Mahler's classes $A, S, T$ and $U$ are defined as follows. Let $\xi$ be a real number. Then
\begin{itemize}
\item $\xi$ is an $A$-number if $w(\xi)=0.$

\item $\xi$ is an $S$-number if $0< w(\xi) < \infty.$

\item $\xi$ is a $T$-number if $w(\xi)=\infty$ and $w_n(\xi) < \infty$ for any $n \geq 1.$

\item $\xi$ is a $U$-number  if  $w(\xi)=\infty$ and $w_n(\xi) = \infty$ for some $n $ onwards.
\end{itemize}

It is known that $A$-numbers are exactly the class of algebraic numbers.
  There is another classification of reals known as Koksma's classification, where  instead of looking
at the approximation of $0$ by integer polynomials evaluated at the real number
$\xi$,  the approximation of $\xi$ by algebraic numbers is considered.

For a real number $\a$ let $\deg(\a)=\deg(P)$ and $H(\a)=H(P),$ where $P(X) \in \mathbb{Z}[X]$ is the minimal polynomial of $\a$.
For a positive integer $n$ and real numbers $\xi$ and $H \geq 1$, define
\[
w^*_{n}(\xi, H)=\min\{|\xi - \a|: \a \text{~real algebraic~} \deg(\a) \leq n, ~H(\a) \leq H,~ \a \neq \xi                \}.
\]
 Then set
\[
w^*_n(\xi) = \limsup_{H \to \infty} \dfrac{-\log (H w^*_n(\xi, H))}{\log H}
\]
and
\[
w^*(\xi) = \limsup_{n \to \infty} \dfrac{w^*_n(\xi)}{n}.
\]
In other words, $w^*_n(\xi)$ is the supremum of the real numbers $w$ for which there
exist infinitely many real algebraic numbers $\a$ of degree at most $n$ satisfying
\[
0< |\xi - \a| < H(\a)^{-w-1}.
\]
 Koksma's  classes $A^*, S^*, T^*$ and $U^*$ are defined as follows. Let $\xi$ be a real number. Then
\begin{itemize}
\item $\xi$ is an $A^*$-number if $w^*(\xi)=0.$

\item $\xi$ is an $S^*$-number if $0< w^*(\xi) < \infty.$

\item $\xi$ is a $T^*$-number if $w^*(\xi)=\infty$ and $w^*_n(\xi) < \infty$ for any $n \geq 1.$

\item $\xi$ is a $U^*$-number  if  $w^*(\xi)=\infty$ and $w^*_n(\xi) = \infty$ for some $n $ onwards.
\end{itemize}
 It is well-know that the classifications of Mahler and of Koksma coincide, in the sense that for any real number $\xi$,
\begin{center}
$\xi$ is an $A$ (resp. $S, T$ or $U$) number $\iff$  $\xi$ is an $A^*$ (resp. $S^*, T^*$ or $U^*$) number.
\end{center}
We now prove the following.
\begin{theorem} (\cite{dean})
\label{mahler}

$(a)$
Assume $r$ is a Cohen real over any transitive model $M$. Then $r$ is in the class $U$.

$(b)$ Assume $r$ is a random real over any transitive model $M$. Then $r$ is in the class $S$.
\end{theorem}

For the proof of Theorem \ref{mahler}$(b)$, we will need the following result of Sprindzuk.
\begin{lemma} $($Sprindzuk 1965$)$
\label{Sprindzuk}
There exists a $G_\delta$ set $A$ of measure zero which contains all transcendental numbers $\xi$ with $w^*(\xi)>1.$
\end{lemma}
\begin{proof}
	We present a proof for completeness. Set $S=\{\xi: \xi \text{~is transcendental and ~}  w^*(\xi)>1   \}$.
	Let $\epsilon>0$ be given. For each $n, H \geq 1$ set
	\[
	X_{n, H}= \{\alpha: \alpha \text{~is real algebraic,~} \deg(\alpha) \leq n \text{~and~} H(\alpha) \leq H\}.
	\]
	
	For each $n$ let $i_n$ be sufficiently large such that
	\[
	\sum_{H=i_n}^\infty n(2H+1)^{n+1} (\dfrac{2}{H^{n+3}}) < \dfrac{\epsilon}{2^{n+1}}.
	\]
	%\dfrac{\epsilon}{2^n}.
	Set
	%$U^\epsilon= \bigcup_{n} T^\epsilon_n$,
	%where
	\[
	A_\epsilon= \bigcup_{n=1}^\infty \bigcup_{H=i_n}^\infty   \bigcup_{y \in A_{n, H}} (y - \dfrac{1}{H^{n+3}}, y+ \dfrac{1}{H^{n+3}}).
	\]
	Then it is easily seen that
	\[
	m(A_\epsilon)= \sum_{n=1}^\infty \sum_{H=i_n}^\infty \sum_{y \in A_{n, H}}  (\dfrac{2}{H^{n+3}}) \leq \sum_{n=1}^\infty \dfrac{\epsilon}{2^{n+1}} <\epsilon.
	\]
	%It is easily seen that
	%\[
	%\mu(U_\epsilon) \leq \sum_{n=1}^\infty \mu(T^\epsilon_n) \leq  \sum_{n=1}^\infty \sum_{H=i_n}^\infty n(2H+1)^{n+1} \dfrac{2}{H^{n+3}}  \leq \sum_{n=1}^\infty %\dfrac{\epsilon}{2^n} = \epsilon.
	%\]
	We now show that $A_\epsilon$ contains  all real numbers $\xi$ with $w^*(\xi) > 1.$
	Given such a $\xi,$ let
	$n$ be large enough such that $ \dfrac{w^*_n(\xi)}{n} > 1 +2/n.$
	Then choose $H > i_n$ such that  $$\dfrac{-\log (H w^*_n(\xi, H))}{\log H} > n(1 + 2/n).$$
	Let $\alpha \in X_{n, H}$ be such that $w^*_n(\xi, H)= |\xi - \alpha|.$ It follows that
	\[
	|\xi - \alpha| < \dfrac{1}{H^{ n(1 + 2/n)+1}} = \dfrac{1}{H^{ n+3}},
	\]
	and hence $\xi \in A_\epsilon.$
	It follows that
	\[
	m(S) \leq m(A_\epsilon) < \epsilon.
	\]
	Finally let $A=\bigcap_{0<n<\omega}A_n$. Then $A$ is a
	$G_\delta$ setof measure zero which contains $S$.
	The result follows.
\end{proof}

 We are now ready to complete the proof of Theorem \ref{mahler}.
\begin{proof}[Proof of Theorem \ref{mahler} $(a)$ ]
 Assume $r$ is a Cohen real. Let $n \geq 1.$ By the proof of Theorem \ref{largeness}$(2)$, we can find $m > n$
 such that $r \cap \{m, \dots, m^2         \}=\emptyset$. Let $\a=\sum_{i < m} \dfrac{r(i)}{2^{i+1}}$,
 where $r(i)=1$ if $i \in r$ and $r(i)=0$ if $i \notin r.$ Then $H(\a) \leq 2^{m}$ and
 \[
 |r - \a| =\sum_{i > m} \dfrac{r(i)}{2^{i+1}} = \sum_{i> m^2}\dfrac{r(i)}{2^{i+1}} \leq \dfrac{1}{2^{m^2+1}} < \dfrac{1}{2^{m^2}}.
 \]
 Then
 \[
 \dfrac{-\log(2^{m}2^{-m^2})}{\log 2^{m}}=\dfrac{\log 2^{m^2-m}}{\log 2^{m}}=\dfrac{m^2-m}{m}=m-1 \geq n.
 \]
It follows that
 \[
 w^*_{1}(\xi) =\limsup_{H \to \infty} \dfrac{-\log (H w^*_1(\xi, H))}{\log H}=\infty.
 \]
Thus $r$ is in class $U$.
\end{proof}
 \begin{proof}[Proof of Theorem \ref{mahler} $(b)$ ]
Suppose $r$ is a random real. Let $p=[T] \in \MRB.$
Set $S=T \cap (2^\omega \setminus A)$, where $A$ is  the set introduced in Lemma \ref{Sprindzuk}.
Then $\mu(S)=\mu(T)>0$, and so $q=[S]$ is well-defined and it extends $p$. Clearly,
\[
q \Vdash~``~\dot{r} \text{~is in the class~}S\text{~''.~}
\]
The result follows.
 \end{proof}
 Recall from number theory that a Liouville number is a real number $r$ with the property that, for every positive integer $n$, there exist infinitely many pairs of integers $(p, q)$ with $q > 1$ such that
 \begin{center}
 $0 < |r - \dfrac{p}{q}| < \dfrac{1}{q^n}$.
 \end{center}
 One can easily show that a real number $r$ is a Liouville number if and only if $w^*_1(r)=\infty.$ It follows from the proof of Theorem \ref{mahler}
 that each Cohen reals is a  Liouville number.
 Let us close this section with the following corollary of the above result.
\begin{corollary}
\begin{enumerate}
\item [(a)] In the generic extension by $\MCB,$ the set $\{r \in \MRB: r$ is Cohen generic over $V           \}$ has measure zero.

\item [(b)] The set $\{r \in \MRB: r$ is a  Liouville number $\}$ contains a perfect set.
\end{enumerate}
\end{corollary}
\begin{proof}
(a) It follows from Lemma \ref{Sprindzuk} that almost all reals are $S$-numbers. As Cohen reals are $U$-numbers, so
$\{r \in \MRB: r$ is Cohen generic over $V           \}$ must have measure zero.\footnote{Indeed, by a result of Kasch and Volkmann \cite{kasch}
	the set of $T$-numbers and
	the set of $U$-numbers
	have  Hausdorff dimension  zero, in particular the set $\{r \in \MRB: r$ is Cohen generic over $V           \}$ has Hausdorff dimension zero.}

(b) Let $\theta=(2^{\aleph_0})^+$ and let  $M$ be a countable elementary submodel of $H(\theta).$  Let $P \subseteq \MRB$ be a perfect set of reals
such that each $r \in P$ is Cohen generic over $M$. Let $r \in P$. By Theorem \ref{mahler}(a), $M[r]\models$`` $r$ is a  Liouville number'', and hence by absoluteness,
$r$ is a  Liouville number. Thus
$\{r \in \MRB: r$ is a  Liouville number $ \} \supseteq T,$
and the result follows.
\end{proof}

%%%%%
%%%%%%%%%%%%

\section{Addition of Cohen or random reals defines $\mathbb{Z}$}
An infinite structure $(M, <, \dots)$ which is totally ordered by $<$ is called an o-minimal structure if and only if every definable subset $X$ of $M$ (with parameters from $M$) is a finite union of intervals and points. It is clear that the set $\mathbb{Z}$ of integers is not definable in an o-minimal structure.

It follows from Tarski's elimination of quantifiers that the structure
$$\mathfrak{R}=(\mathbb{R}, +, \cdot, <, 0, 1)$$
is an o-minimal structure and hence  $\mathbb{Z}$ is not definable in it. An important problem is to see which extensions of
$\mathfrak{R}$ remain o-minimal or at least retain the property that $\mathbb{Z}$, the set of integers, is not definable in them. There is a vast range of results in this direction; see for example \cite{forn}, \cite{hier}, \cite{khani}, \cite{miller} ,  \cite{dries1} and  \cite{wilkie}.

Recall that a subset $C \subseteq \mathbb{R}$ is called tame if $\mathbb{Z}$ is not definable in the structure $(\mathfrak{R}, C)$ and it is called wild otherwise.

We relate the above phenomenon to forcing and show that the addition of Cohen or random reals to the structure $\mathfrak{R}$, when considered as unary predicates,  defines $\mathbb{Z}$ which completes the proof of Theorem \ref{o-minimality}.
It follows that the Cohen and random reals are wild in the above sense.

\begin{remark}
It is known that  o-minimality is closed under expansions by constants. It follows that if we consider Cohen and random reals as constants (or any other generic real), then
the structure $(\mathbb{R}^{V[r]}, +, \cdot, <, 0, 1, r)$ remains o-minimal and hence it does not define $\mathbb{Z}.$
\end{remark}
Accordingly, in what follows we interpret a Cohen or random real as a subset of $\omega$, i.e., as a unary predicate in $\mathbb R$. This interpretation works in the obvious way: $r \in 2^\omega$ is identified with the subset of $\omega$ for which it is the characteristic function.

First we prove  Theorem \ref{o-minimality} for Cohen reals.
Thus assume $r$ is a Cohen real over $V$ and in $V[r],$ consider the structure
$(\mathbb{R}^{V[r]}, +, \cdot, <, 0, 1, r)$. Then it is easily seen, by density arguments, that
\[
\mathbb{Z}= \{x \in \mathbb{R}^{V[r]}: \exists a, b \in r \text{~such that ~} x=a - b     \},
\]
and hence $\mathbb{Z}$ is definable in $(\mathbb{R}^{V[r]}, +, \cdot, <, 0, 1, r)$.

For a random real, the theorem follows from the following easy lemma.
\begin{lemma}
Almost all reals define $\mathbb{Z}$ in the sense that for almost all reals $r \in 2^\omega,$ the set  $\mathbb{Z}$ is definable in $(\mathbb{R}, +, \cdot, <, 0, 1, r)$.
\end{lemma}
\begin{proof}
For any $m \in \omega,$ the set
\[
A_m=\{ r \in 2^\omega: \exists a, b \in \omega \text{~such that~} r(a) = r(b) = 1 \text{~and~} m=  a - b                 \}
\]
is of measure one, and hence so is the set $A= \bigcap_{m \in \omega}A_m.$ But any $r \in A$ defines $\mathbb{Z}$, as
$\mathbb{Z}= \{x \in \mathbb{R}: \exists a, b \in r \text{~such that~} x=a - b     \}.$
\end{proof}

To complement these observations, we now define a forcing notion which adds a generic real $r$ that does not define $\mathbb{Z}$. Let
$\MPB$ be the set  of all finite functions $p: n \to 2$ such that $p(k)=0$ if $k< n$ is not a power of $2$ (i.e., it is not of the form $2^i$ for some $i$),
 ordered by reverse inclusion. Let $G$ be $\MPB$-generic over $V$ and let $r$ be the generic real added by $G$, namely
\[
r=\{k< \omega: \exists p \in G \text{~such that~} p(k)=1          \}.
\]
Note that $r \subseteq 2^{\mathbb{Z}}=\{2^n: n \in \mathbb{Z}  \}$. In order to show that $\mathbb{Z}$ is not definable in $(\mathbb{R}^{V[r]}, +, \cdot, <, 0, 1, r)$,
we need the following theorem.
\begin{theorem}
\label{friedman-miller theorem}
$\ $
\begin{enumerate}
 \item [(a)] (Friedman and Miller \cite{friedman-miller})
Let $\mathfrak{R}$ be an o-minimal expansion of $(\mathbb{R}, <, +)$ and $E\subseteq \mathbb{R}$
be  such that for every $n \in \mathbb{N}$ and definable function $f: \mathbb{R}^n \to \mathbb{R}$, the image $f(E^n)$ has no interior.
If every subset of $\mathbb{R}$ definable in $(\mathfrak{R}, E)$ has interior or is
nowhere dense, then every subset of $\mathbb{R}$ definable in $(\mathfrak{R}, E)^\sharp$ \footnote{$(\mathfrak{R}, E)^\sharp$
 is the structure obtained by adding to $\mathfrak{R}$
 predicates picking out every subset of every cartesian power $E^k$ of $E$.} has interior or
is nowhere dense. The same holds true with ``nowhere dense'' replaced by
any of ``null'', ``countable'', ``a finite union of discrete sets'' or ``discrete''.

\item [(b)] (Van den Dries \cite{dries1}) Every subset of $\mathbb{R}$ definable in $(\mathbb{R}, <, +, ., 0, 1, 2^{\mathbb{Z}})$
is the union of an open set and finitely
many discrete sets.
\end{enumerate}
\end{theorem}
It follows from the above theorem that $\mathbb{Z}$ is not definable  in $(\mathbb{R}^{V[r]}, +, \cdot, <, 0, 1, 2^{\mathbb{Z}})^\sharp$,
and hence it is not definable in $(\mathbb{R}^{V[r]}, +, \cdot, <, 0, 1, r)$ either.

We note that some (but not all) Silver-generic reals $r$ have the property that $r \subseteq 2^{\mathbb Z}$. (This property is forced by a condition in the Silver poset, namely the partial function mapping $\omega \setminus 2^{\mathbb Z}$ to $0$.) Similarly, some (but not all) Sacks-generic reals $r$ have the property that $r \subseteq 2^{\mathbb Z}$. (This is forced by a condition in the Sacks poset, $\{ r \in 2^\omega :\, r(k) = 0 \text{~whenever~}k \notin 2^\mathbb Z\}$.) Consequently, at least some Silver-generic and Sacks-generic reals have the property that $(\mathbb{R}^{V[r]}, +, \cdot, <, 0, 1, r)$ does not define $\mathbb{Z}$.

\begin{question}
If $r$ is an arbitrary Silver-generic real over $V$, does $(\mathbb{R}^{V[r]}, +, \cdot, <, 0, 1, r)$ define $\mathbb{Z}$? What about Sacks-generic reals? Other minimal reals?
\end{question}

Recall that Silver reals and Sacks reals are two primary examples of minimal reals: i.e., reals $r$ generic over $V$ such that if $s \in V[r]$ then either $s \in V$ or $V[s] = V[r]$.

\begin{question}
Assume $V[r]$ is a generic extension of $V$ by Cohen (res. random) forcing  and (in $V[r]$) set
\begin{center}
$C= \{ s \in \mathbb{R}: s$ is Cohen (res. random) generic over $V              \}$.
\end{center}
Does the structure  $(\mathbb{R}^{V[r]}, +, \cdot, <, 0, 1, C)$ define $\mathbb{Z} ?$.
\end{question}
\begin{remark}
It is evident that the structure $(\mathbb{R}^{V[r]}, +, \cdot, <, 0, 1, C)$ is not o-minimal, as the set $C$ is not  a finite union of intervals and points.
\end{remark}

\section{Adding many random reals may add many Cohen reals}
It is a well-know fact that forcing with $\MRB \times \MRB$ adds a Cohen real; in fact, if $r_1, r_2$
are the added random reals, then $c=r_1 + r_2$ is Cohen \cite{judah}. This in turn implies all reals $c+a,$ where $a \in \MRB^V,$ are Cohen,
and so, we have continuum many Cohen reals over $V$. However, the sequence $\langle  c+a: a \in \MRB^V   \rangle$
fails to be $\MCB((2^{\aleph_0})^V)$-generic over $V$. In fact, there is no sequence $\langle  c_i: i < \omega_1 \rangle \in V[r_1, r_2]$
of Cohen reals which is $\MCB(\omega_1)$-generic over $V$.

The usual proof of the above fact is based on the characterization of Cohen reals given by Lemma \ref{cohen characterization}.
In this section,
we give a proof of Theorem \ref{adding many} which is direct and avoids any characterization of Cohen reals. As in the classical case, the Steinhaus's lemma plays an essential role in our proof.

A famous theorem of Steinhaus \cite{steinhaus} from $\text{1920}$ asserts that if $A, B \subseteq \mathbb{R}^n$ are measurable sets with positive
Lebesgue measure, then  $A + B$ has an interior point; see also \cite{stromberg}.  Here, we
need a version of Steinhaus theorem  for  the space
$2^{\k \times \omega}$.

For $S, T \subseteq 2^{\k \times \omega}$, set $S+T=\{x+y: x \in S$ and $ y \in T              \}$, where
$x+y: \k \times \omega \to 2$ is defined by
$$(x+y)(\a,n) = x(\a,n)+y(\a,n)~(\text{mod} ~2).$$
Note that the above addition operation is continuous.
\begin{lemma}
	\label{steinhaus lemma 1}
	Suppose $S \subseteq 2^{\k \times \omega}$ is Borel and non-null. Then $S-S$ contains an open set around the zero function $0$.
\end{lemma}
\begin{proof}
	We follow \cite{stromberg}. Set $\mu=\mu_\k$ be the product measure on $2^{\k \times \omega}$. As $S$ is Borel and non-null, there is a compact subset of $S$ of positive $\mu$-measure, so may suppose that $S$ itself is compact. Let $U \supseteq S$
	be an open set with $\mu(U) < 2\cdot \mu(S).$
	By continuity of addition, we can find an open set $V$ containing the zero function $0$ such that $V+S \subseteq U.$
	
	We show that $V \subseteq S - S$. Thus suppose $x \in V.$ Then $(x+S) \cap S \neq \emptyset,$ as otherwise we will have
	$(x+S) \cup S \subseteq U,$ and hence $\mu(U) \geq 2\cdot \mu(S),$ which is in contradiction with our choice of $U$.
	Thus let $y_1, y_2 \in S$ be such that $x+ y_1 = y_2.$ Then $x= y_2 - y_1 \in S - S$ as required.
\end{proof}
Similarly, we have the following:
\begin{lemma}
	\label{steinhaus lemma}
	Suppose $S, T \subseteq 2^{\k \times \omega}$ are Borel and non-null. Then $S+T$ contains an open set.
\end{lemma}
Suppose $S, T \subseteq 2^{\k \times \omega}$ are Borel and non-null. It follows from Lemma \ref{steinhaus lemma} that for some $p \in \MCB(\k),$
$\tr{p} \subseteq S+T.$ Thus, by continuity of the addition, we can find
$x \in S$ and $y \in T$ such that:
\begin{itemize}
	\item $(x+y) \upharpoonright \dom(p) = p.$
	\item The sets $S \cap \tr{x \upharpoonright \dom(p)}$ and $T \cap \tr{y \upharpoonright \dom(p)}$ are Borel and non-null.
\end{itemize}
We now complete the proof of Theorem \ref{adding many}.  Thus force with $\MRB(\k) \times \MRB(\k)$
and let $G \times H$ be generic over $V$. Let $\langle \langle r_\a: \a < \k \rangle, \langle  s_\a: \a < \kappa   \rangle \rangle$
be the sequence of random reals added by $G \times H.$

For $\a < \k$ set $c_\a=r_\a+s_\a$. The following completes the proof:
\begin{lemma}
	The sequence $\langle  c_\a: \a < \k      \rangle$ is a sequence of $\kappa$-Cohen reals over $V$.
\end{lemma}
\begin{proof}
	%By Lemma \ref{characterization of cohen reals}, it suffices to show that if $I \in V, I \subseteq \kappa$ is countable,
	%then $\langle  c_\a: \a \in I  \rangle$ is $\MCB(I)$-generic over $V$. Thus fix $I$ as above.
	It suffices to prove the following:
	
	$\hspace{1.5cm}$ For every $([S], [T])\in {\MRB}(\k) \times \MRB(\k)$,
	and every open dense subset $D\in V$

	$(*)$ $\hspace{0.9cm}$  of ${\MCB}(\k)$,~ there is
	$([\ov{S}], [\ov{T}])\leq ([S], [T])$ such~ that $([\ov{S}], [\ov{T}]) \vdash
	``  \lan \lusim{c}_\a : \a\in \k\ran$
	
	$\hspace{1.5cm}$  extends some element of $D$''.
	
	Thus fix $([S], [T])\in {\MRB}(\k) \times \MRB(\k)$ and $D \in V$ as above, where $S, T \subseteq 2^{\k \times \omega}$ are Borel and non-null.
	By Lemma \ref{steinhaus lemma} and the remarks after it, we can find
	$p \in \MCB(\k)$ and $(x, y) \in S \times T$ such that:
	\begin{enumerate}
		\item $\tr{p} \subseteq S+T.$
		\item $(x+y) \upharpoonright \dom(p) = p.$
		\item The sets $S \cap \tr{x \upharpoonright \dom(p)}$ and $T \cap \tr{y \upharpoonright \dom(p)}$ are Borel and non-null.
	\end{enumerate}
	Now let $q \in D$ be such that
	\begin{center}
		$([S \cap \tr{x \upharpoonright \dom(p)}], [T \cap \tr{y \upharpoonright \dom(p)}]) \Vdash$``$q \leq_{\MCB(\k)} p$''.
	\end{center}
	Using continuity of the addition and further application of
	Lemma \ref{steinhaus lemma} and the remarks after it, we can find
	$x', y' $ such that:
	\begin{enumerate}
		\item [(4)] $x' \in S \cap \tr{x \upharpoonright \dom(p)}$ and $y' \in T \cap \tr{y \upharpoonright \dom(p)}.$
		\item [(5)] $(x'+y') \upharpoonright \dom(q) = q.$
		\item [(6)] The sets $S \cap \tr{x' \upharpoonright \dom(q)}$ and $T \cap \tr{y' \upharpoonright \dom(q)}$ are Borel and non-null.
	\end{enumerate}
	It is now clear that
	\begin{center}
		$([S \cap \tr{x' \upharpoonright \dom(q)}], [T \cap \tr{y' \upharpoonright \dom(q)}]) \Vdash$``$ \lan \lusim{c}_\a : \a\in \k\ran$ extends $q$''.
	\end{center}
	The result follows.
\end{proof}

Will Brian,
Department of Mathematics and Statistics University of North Carolina at Charlotte Charlotte,
NC 28223-0001, USA

E-mail address: wbrian.math@gmail.com

URL: http://wrbrian.wordpress.com

Mohammad Golshani,
School of Mathematics, Institute for Research in Fundamental Sciences (IPM), P.O. Box:
19395-5746, Tehran-Iran.

E-mail address: golshani.m@gmail.com

URL: http://math.ipm.ac.ir/golshani/

\end{document}